\documentclass[10pt]{amsart}
\usepackage{amsmath,amssymb,amsthm}
\usepackage[dvipdfmx]{graphicx}
\usepackage{color,comment,url}
\newtheorem{theorem}{Theorem}    
\newtheorem{lemma}{Lemma} 
\newtheorem{claim}{Claim}

\newtheorem{corollary}{Corollary}
\theoremstyle{definition}
\newtheorem{definition}[theorem]{Definition}

\newtheorem{remark}[theorem]{Remark}
\newtheorem*{remark*}{Remark}
\newcommand{\Z}{\mathbb{Z}}
\newcommand{\R}{\mathbb{R}}

\newcommand{\ds}{\displaystyle}

\title{A note on HOMFLY polynomial of positive braid links}
\author[T.Ito]{Tetsuya Ito}
\address{Department of Mathematics, Kyoto University, Kyoto 606-8502, JAPAN}
\email{tetitoh@math.kyoto-u.ac.jp}

\begin{document}

\begin{abstract}
For a positive braid link, a link represented as a closed positive braids, 
we determine the first few coefficients of its HOMFLY polynomial in terms of geometric invariants such as, the maximum euler characteristics, the number of split factors, and the number of prime factors. Our results give improvements of known results for Conway and Jones polynomial of positive braid links. In Appendix, we present a simpler proof of theorem of Cromwell, a positive braid diagram represent composite link if and only if the the diagram is composite.
\end{abstract}

\maketitle

\section{Introduction}

A knot or link $K$ in $S^{3}$ is a \emph{positive braid knot/link} (or,  \emph{braid positive}) if it is represented by the closure of a positive braid. Reflecting the positivity of braids, various knot invariants such as the signature\footnote{Here we adapt the convention for the signature opposite to Knotinto \cite{K} so that the right-handed trefoil, the closure of the positive $2$-braid $\sigma_1^{3}$, has signature $2$.}, or the Conway polynomial are positive for positive braid links. Here we say that a polynomial is positive if all the coefficients are non-negative. 

However, one should view the positivity of these invariants as a consequence of positivity of knot diagrams, not the positivity of braids since the same properties hold for an almost positive knot\footnote{Here we regard a positive knot, a knot that admits a diagram without negative crossing, is a special case of almost positive knots.}, a knot that admits a diagram with at most one negative crossing \cite{Cr,PT}.
 
In this note we observe a positivity of the HOMFLY polynomial for braid positive links after suitable normalization, and provide various additional information concerning its first few top coefficients. 
Let $P_K(v,a)$ be the HOMFLY polynomial of a knot or link $K$ defined by the skein relation\footnote{Here we use the convention adapted in Knotinfo \cite{K}.}
\[ v^{-1} P_{\raisebox{-2mm}{\includegraphics*[width=3.5mm]{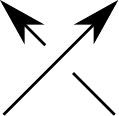}}}(v,z) -  vP_{\raisebox{-2mm}{\includegraphics*[width=3.5mm]{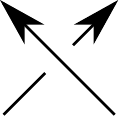}}}(v,z)= z P_{\raisebox{-2mm}{\includegraphics*[width=3.5mm]{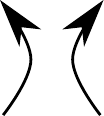}}}(v,z), \quad P_{\sf Unknot}(v,z) =1. \]
For a link $K$, we put
\begin{itemize}
\item $\#K = $ the number of components of $K$.
\item $\chi(K) = $ the maximal euler characteristic for (possibly non-connected) Seifert surface of $K$.
\item $s(K) =$ the number of split factors of $K$.
\item $p(K) =$ the number of prime factors of $K$.
\end{itemize}

Here we use the following convention for $p(K)$. Every link $K$ is a split union of non-split links as $K=K_1 \sqcup \cdots \sqcup K_{s(K)}$. For each $K_i$ we define $p(K_i)$ by
\[ p(K_i)=\begin{cases}\max \{n \: | \: K_i=K_i^1\# K_i^2 \# \cdots \# K_i^n,\ K^j_i \neq \mbox{unknot}\}, & K \neq \mbox{unknot} \\
0 & K = \mbox{unknot} \end{cases}\]
and define $p(K)=p(K_1)+\cdots+p(K_{s(K)})$. Thus in our definition, $p(\mathsf{Unlink})=0$.

\begin{definition}
We define \emph{the normalized HOMFLY polynomial} of a link $K$ by
\begin{align*}
\widetilde{P}_{K}(\alpha,z) & = (1+\alpha)^{-s(K)+1}(-\alpha)^{-\frac{-\chi(K)+2-\#K}{2}}(v^{-1}z)^{\#K-1}P_{K}(v,z)|_{-v^{2}=\alpha}  \in \Z[\alpha^{\pm 1},z^2]
\end{align*}
\end{definition}

When $K$ is a knot, the normalized HOMFLY polynomial is simply written as 
\[ \widetilde{P}_{K}(\alpha,z) =(-\alpha)^{-g(K)}P_{K}(v,z)|_{-v^{2}=\alpha}.\]

We show that for braid positive links, the normalized HOMFLY polynomial is positive and its first few top coefficients are determined by geometric invariants $p(K),s(K)$, and $\chi(K)$.

\begin{theorem}
\label{theorem:main}
Assume that $K$ is braid positive. Let $m(K)=-\chi(K)+s(K)$ and $d=d(K)=\frac{-\chi(K)+\# K}{2}$.
\begin{itemize}
\item[(i)] $\widetilde{P_K}(\alpha,z) \in \Z[\alpha,z^{2}]$ and $\widetilde{P_K}(\alpha,z)$ is positive.
\item[(ii)] Let $\ds \widetilde{P_K}(\alpha,z)=\sum_{i,j\geq 0} h_{i,j}(K)
\alpha^{i} z^{2j}$. 

\begin{itemize}
\item[(a)] $h_{i,j}(K)=0$ whenever $i+j>d(K)$.
\item[(b)] $\ds h_{i,d-i}(K)=\binom{p(K)}{i}$. 
\item[(c)] $\ds h_{0,d-1}(K)=m(K)$.
\item[(d)] $\ds h_{0,d-2}(K) = \frac{(m(K)-1)(m(K)-2)}{2}+p(K)-1$.
\item[(e)] $\ds (m(K)-2)p(K) \leq h_{1,d-2}(K) \leq (m(K)-2)p(K)+m(K)$.
\item[(f)] $\ds h_{0,d-3}(K)=\frac{(m(K)-1)(m(K)-2)(m(K)-6)}{6}+h_{1,d-2}(K)-2(p(K)-1)$.
\end{itemize}
\end{itemize}
\end{theorem}

\begin{remark}
The inequality (e) is sharp; For the $(2,k)$-torus knot/link, $h_{1,d-2}(T_{2,k})=k-3 = m(T_{2,k})-2$, and the connected sum of $k$ Hopf links 
$H_k=\#_{k}T_{2,2}$, $h_{1,d-2}(H_{k})=k(k-1)=(m(H_k)-2)p(k)+m(H_k)$. 
\end{remark}

The positivity (i)\footnote{Essentially the same positivity phenomenon of HOMFLY polynomial of positive braid links was proven in \cite[Theorem 2.2]{FW} in a different formulation.} reflects the positivity of braids since positive links do not have this positivity in general.

Theorem \ref{theorem:main} improves various known results for braid positive links. Here we state theorem for knot case for a sake of simplicity. First we give a more concrete formula of HOMFLY polynomial of prime positive braid knots.

\begin{corollary}
If $K$ is a prime positive braid knot other than unknot,
then 
\begin{align*}
P_K(v,z) &= v^{2g}z^{2g}\\
& +(2gv^{2g}-v^{2g+2})z^{2g-2}\\
& +((2g-1)(g-1)v^{2g} - h(K)v^{2g+2})z^{2g-4} \\
& +((\frac{(2g-1)(g-1)(2g-6)}{3}+h(K))v^{2g}+(\mbox{higher } v \mbox{ degree terms}))z^{2g-6}\\
& + \mbox{(lower } z \mbox{ degree terms)}
\end{align*}
where $g=g(K)$ and $h(K)=h_{1,g-2}(K)$ satisfies $2g-2\leq h(K) \leq 4g-2$.
\end{corollary}

Let $\nabla_K(t)=\sum_{i}a_{2i}z^{2i}$ be the Conway polynomial of $K$.
Since $\nabla_K(t)=P_{K}(1,z)$ we have the following.

\begin{corollary}
\label{cor:Conway}
If $K$ is a braid positive knot, 
$a_{2g-2}(K)=2g(K)-p(K)$, and
\[ 2g(K)^{2}-(5+2p(K))g(K)+3p(K) \leq a_{2g-4}(K) \leq 2g(K)^{2}-(3+2p(K))g(K)+\frac{p(K)(p(K)+5)}{2}.\]
In particular, for a prime braid positive knot $K$,
\[ a_{2g-2}(K)=2g(K)-1,\quad 2g(K)^{2}-7g(K)+3 \leq a_{2g-4}(K) \leq 2g(K)^{2}-5g(K)+3.\]
\end{corollary}

This improves the inequalities of $a_{2g-2}, a_{2g-4}$ 
\[  g(K) \leq a_{2g-2}(K) \leq 2g(K)-1, \quad \frac{g(K)(g(K)-1)}{2} \leq a_{2g-4}(K) \leq 2g(K)^{2}-5g(K)+3 \]
proven in \cite{vB}.

Similarly, since the Jones polynomial $V_K(t)$ is obtained from the HOMFLY polynomial as $V_K(t)=P(t,t^{1/2}-t^{-1/2})$ we have the following. 

\begin{corollary}
\label{cor:Jones}
If $K$ is a non-split positive braid knot,
then 
\[t^{g(K)} V_{K}(t)= 1 + p(K)t^{2}+k(K)t^{3}+\mbox{(higher order terms)} \]
and the third coefficient $k(K)=h(K)+(1-2g(K))p(K)$ satisfies the inequality
\[ -p(K) \leq k(K)\leq -p(K)+2g(K). \]
\end{corollary}

This improves the inequality (upper bound) of $k(K)$ 
\[ -p(K) \leq k(K) \leq \frac{3}{2}(-p(K)+2g(K)) \]
proven in \cite{Sto}.

One of our interests of braid positive knots comes from an L-space knot, a knot in $S^{3}$ that admits a (positive) L-space surgery. An L-space knot is prime \cite{Kr1}, fibered \cite{Ni} and strongly quasipositive \cite{He} (see also \cite{BS}). Here a knot $K$ is \emph{strongly quasipositive} if it is represented by the closure of $n$-braid which are product of positive band generators $\sigma_{i,j}=(\sigma_{i+1}\cdots \sigma_{j-1})\sigma_{j}(\sigma_{i+1}\cdots \sigma_{j-1})^{-1}$ ($1\leq i<j\leq n-1$) for some $n$. Although in general a fibered strongly quasipositive knot is not braid positive, currently all known examples of hyperbolic L-space knots are braid positive so it is interesting to compare properties of L-space knots and braid positive knots.

The (symmetrized) Alexander polynomial $\Delta_K(t)=\nabla_K(t^{1/2}-t^{-1/2})$ of  L-space knots has various special features \cite{OS,HW};
there is a sequence of integers $0<n_1<\cdots <n_{k-2}<n_{k-1}=g(K)-1<n_{k}=g(K)$ such that
\[
\Delta_K(t) = \sum_{i=1}^{k}(-1)^{k+i}(t^{n_k}+t^{-n_{k}}) +1=  (t^{g(K)}+t^{-g(K)})-(t^{g(K)-1}+t^{-(g(K)-1)})+ \cdots.
\]
and there are more constraints for sequence $0<n_1<\cdots <n_{k}$ \cite{Kr2}.
On the other hand, by Corollary \ref{cor:Conway}, a prime positive braid knot has the symmetrized Alexander polynomial of the form
\[ \Delta_K(t)=(t^{g(K)}+t^{-g(K)})-(t^{g(K)-1}+t^{-(g(K)-1)})-\alpha_3(K)(t^{g(K)-2}+t^{-(g(K)-2)})+\cdots \]
where
\[ -2g(K)+1\leq \alpha_3(K)=-h(K)+2g(K)-1 \leq 1 \]
Note that for a prime braid positive knot $K$, the third coefficient $\alpha_3(K)$ of the symmetrized Alexander polynomial of $K$ is equal to $-k(K)$, the minus of the third coefficient of the Jones polynomial. This gives the following constraint for Jones/HOMFLY polynomial of L-space positive braid knots.

\begin{corollary}
\label{cor:L-spaceknot}
If $K$ is an L-space positive braid knot, then 
\begin{itemize}
\item[(i)] $t^{-g(K)} V_{K}(t)= 1 + t^{2}+k(K)t^{3}+\mbox{(higher order terms)}$,
and the third coefficient $k(K)$ is either $0$ or $-1$.
\item[(ii)]  $h(K)=h_{1,g-2}(K)$ is either $2g(K)-2$ or $2g(K)-1$.
\end{itemize}
\end{corollary}

It is an interesting question to ask whether the Jones/HOMFLY polynomial of hyperbolic L-space knot shares the same properties of positive braid knots.

\begin{remark}
The $(2,3)$-cable of $(2,3)$-torus knot $K$ is a non-hyperbolic $L$-space knot which is not braid positive. $K$ is represented by a closure of a $4$-braid $(\sigma_2 \sigma_1 \sigma_3 \sigma_2)^{3}\sigma_1^{-3}$. The normalized HOMFLY polynomial is
\[ \widetilde{P_K}(\alpha,z)= (3 - \alpha -\alpha^{2} -2\alpha^{3})
+(9-5\alpha^{2}-\alpha^{3})z^{2}+(6-\alpha^{2})z^{4}+z^{6}
\]
and the Jones polynomial is
\[ t^{-3}V_K(t) = 1+t^{3}-t^{7}-t^{9}+t^{10}. \]
Thus one cannot expect the HOMFLY or Jones polynomial of general L-space knot has properties similar to positive braid knots. 
\end{remark}

\begin{remark}
After the first version of the paper appeared, Baker and Kegel informed
me of an example of a hyperbolic L-space knot which is not braid positive \cite{bk}.
Their example is the closed 4-braid $(\sigma_{2}\sigma_1\sigma_3\sigma_2)^3 \sigma_{1}^{-1}\sigma_2 \sigma_1^2 \sigma_2$. Its normalized HOMFLY polynomial is not positive, but it satisfies all the properties (a)--(f) in Theorem \ref{theorem:main}. Moreover, the Jones and HOMFLY polynomial of this example satisfies the property stated in Corollary \ref{cor:L-spaceknot}. 
It seems to be possible that Corollary \ref{cor:L-spaceknot} holds for hyperbolic L-space knots.  
\end{remark}

\section*{Acknowledgement}
The author has been partially supported by JSPS KAKENHI Grant Number
19K03490, 16H02145. He would like to thank K. Baker and M. Kegel for informing
me of an example of non-braid-positive hyperbolic L-space knots and for stimulating discussions.

\section{Proof}

\subsection{Review of properties of positive braid links}

Before starting the proof, we recall various special properties of positive braid links. 

In the following, by abuse of notation we often confuse a positive braid $\beta$ and its particular word representative $w$.
(For example, by a diagram $D_{\beta}$ of $K$ determined by a positive braid representative $\beta$, we actually mean the diagram $D_w$ obtained by taking a particular positive braid word representative $w$ of $\beta$.)

First of all, if $K$ is the closure of a positive $n$-braid $\beta$, by Bennequin's inequality 
\[ \chi(K)=n-e(\beta). \]
Here $e(\beta)$ denotes the exponent sum of $\beta$. Thus one can read $\chi(K)$ from a positive braid representative.

Actually, one can also read $s(K)$ and $p(K)$ from a positive positive braid representative. Let $K$ be a link in $\R^{3}$. 
Assume that the natural projection $\pi:\R^{3} \rightarrow \R^{2}$ ($\pi(x,y,z)=(x,y)$) gives a knot diagram $D=\pi(K)$.

\begin{definition}
A link diagram $D$ is 
\begin{itemize}
\item[--] \emph{irreducible}, if there is no crossing $c$ such that $D \setminus c$ is disconnected.
\item[--] \emph{split}, if there is a circle $c \subset \R^{2}$ which is disjoint from $D$, such that $c$ separates $\R^{2}$ into two connected components $U,V$ so that both $U\cap D$ and $V \cap D$ are non-empty. We call such a circle $c$ a \emph{splitting circle} of $D$.
\item[--] \emph{composite}, if there is a circle $c \subset \R^{2}$ which
 transversely intersects with $D$ at two non-double points, such that $c$ separates $\R^{2}$ into two connected components $U,V$ so that both $U\cap D$ are $V \cap D$ are not an embedded arc. We call such a circle $c$ a \emph{decomposing circle} of $D$.
\end{itemize}
\end{definition}

Obviously, if a diagram $D$ is split/composite then so is $K$.
The following theorem states that the converse is true.

\begin{theorem}\cite{Cr2}
\label{theorem:pbraid-visible}
Let $K$ be a positive braid link and let $D_{\beta}$ be a diagram of $K$ given by a positive braid representative $\beta$.
\begin{itemize}
\item[(i)] If $D_{\beta}$ is irreducible, $K$ is split if and only if $D_{\beta}$ is split.
\item[(ii)] If $D_{\beta}$ is irreducible and non-split, $K$ is non-prime if and only if $D_{\beta}$ is composite.
\end{itemize}
\end{theorem}
More generally, in \cite{Oz} Ozawa proved the same result for positive knots.

Although (i) is easy to see by looking at the linking numbers, the proof of (ii) is more complicated. We give a simplified proof in a spirit of Cromwell's original proof in Appendix.

Note that the assumption that $D_{\beta}$ is irreducible is always satisfied when
$\beta$ is a \emph{minimum positive braid representative}, which we mean that the number of strands of $\beta$ is minimum among all the positive braid representatives of $K$.

Finally, for a positive braid link one can apply the skein relation so that the resulting links are also braid positive.

\begin{theorem}\cite[Lemma 2]{vB}
\label{theorem:resolution}
Let $K$ be positive braid link which is not unlink.
Then there exists a positive braid $\beta$ such that $K$ is a closure of a positive braid of the form $\sigma_j^{2}\beta $. Moreover, such a positive braid representative $\sigma_{j}^2\beta$ of $K$ can be taken so that it is a minimum positive representative.
\end{theorem}

This allows us to use induction in the realm of positive braid links. 
In Appendix, we attach a proof of this fact, as a byproduct of our technical lemma \ref{lemma:technical}, although it is essentially the same as the proof presented in \cite{vB}.

\subsection{Proof of Theorem \ref{theorem:main}}

\begin{proof}[Proof of Theorem \ref{theorem:main}]

We prove theorem by induction on $m(K)=-\chi(K)+s(K)$. 
$m(K)=0$ if and only if $K$ is unlink. For unlnk $K$, $P_{K}(v,z)=\left(\frac{v^{-1}-v}{z}\right)^{s(K)-1}$. Hence $\widetilde{P_K}(\alpha,z)=1$ whenever $m(K)=0$.

Assume that $m(K)>0$. By Theorem \ref{theorem:resolution},
there is a minimum positive braid representative $\beta$ of $K$ of the form
$\beta=\sigma_j^{2}\beta'$ where $\beta'$ is a positive braid.

Let $K_-$ and $K_0$ be the closure of braids $\beta'$, $\sigma_j \beta'$, respectively. Let 
\[ \delta=\frac{1}{2}(\# K- \# K_{0} +1).\]
That is, we define $\delta=0$ if two strands at the first two crossings $\sigma_j^2$ belong to the same component of $K$, and we define $\delta=1$ otherwise.
By the skein relation of the HOMFLY polynomial we have the following skein relation for the normalized HOMFLY polynomial.
\begin{equation}
\label{eqn:skein-base}
 (\alpha+1)^{s(K)-1}\widetilde{P}_{K}(\alpha,z) = (\alpha+1)^{s(K_-)-1}\widetilde{P}_{K_-}(\alpha,z) +z^{2\delta}(1+\alpha)^{s(K_0)-1}\widetilde{P}_{K_0}(\alpha,z).
\end{equation}

Since 
\[ d(K_0)= \begin{cases} \frac{\#K+1+ e(\beta)-1-n}{2} & (\delta=0)\\
\frac{\#K-1+ e(\beta)-1-n}{2} & (\delta=1)\\
\end{cases}
\]
we conclude 
\begin{equation}
\label{eqn:d}
d(K)=d(K_0)+\delta=d(K_-)+1
\end{equation}

By Theorem \ref{theorem:pbraid-visible} we can directly read $s(K)$ and $p(K)$ from the diagram. To keep track of how $s(K)$ and $p(K)$ change we divide various cases of positive braid diagrams. 

Let us write $\beta=\sigma_j^{2}\beta'$ as
\begin{equation}
\label{eqn:form}
\beta = \sigma_{j}^{a_0}A_0 B_0 \sigma_j^{a_1} A_1 B_1 \sigma_j^{a_2} A_2 B_2 \cdots \sigma_j^{a_k} A_k B_k
\end{equation}
where
\[ a_0\geq 2, a_{1},\ldots a_{k}> 0, A_i \in \langle \sigma_1,\ldots, \sigma_{j-1}\rangle, B_i \in \langle \sigma_{j+1},\ldots, \sigma_{n-1}\rangle. \]
(see Figure \ref{fig:rep1}).

Among a positive braid representative of the form (\ref{eqn:form}), we take one so that $j$ is as small as possible, and then take $a_0$ is as large as possible. 

\begin{figure}[htbp]
\includegraphics*[width=90mm,bb=148 609 462 715]{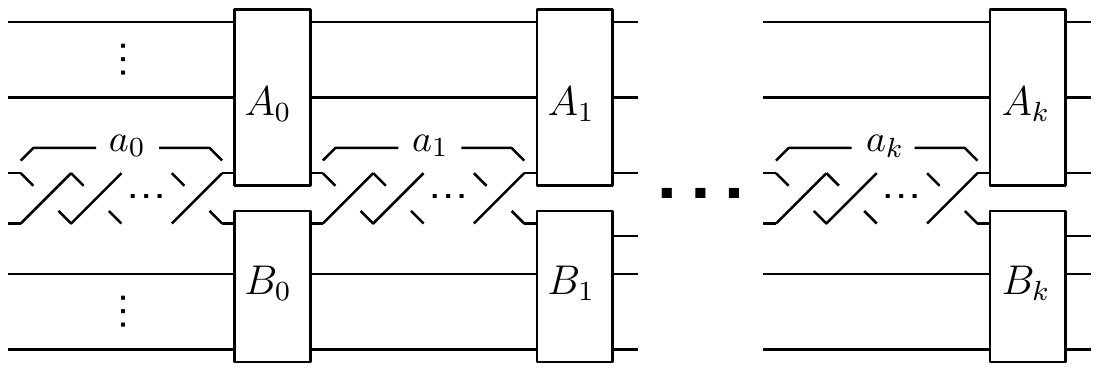}
\caption{Positive braid representative $\beta$} 
\label{fig:rep1}
\end{figure} 

\noindent
\underline{\bf Case 1: $a_0 >3 $}\\

In this case $s(K)=s(K_-)=s(K_0)$ and $p(K)=p(K_-)=p(K_-)$, hence $m(K)=m(K_-)+2=m(K_0)+1$.
All the assertions follow from the standard induction arguments, as we illustrate below.

By the skein relation (\ref{eqn:skein-base}) 
\[ h_{i,j}(K)= h_{i,j}(K_-) + h_{i,j-\delta}(K_0). \]
Therefore by (\ref{eqn:d})
\begin{align*}
h_{i,d(K)-j}(K) &= h_{i,d(K_-)+1-j}(K_-) + h_{i,(d(K_0)+\delta)-\delta-j}(K_0)\\
& = h_{i,d(K_-)-(j-1)}(K_-) + h_{i,d(K_0)-j}(K_0)
\end{align*}
This immediately shows (i).

We confirm assertions (ii) (a)--(d) for $K$;
\begin{align*}
h_{i,d(K)-j}(K)& = h_{i,d(K_-)-(j-1)}(K_-) + h_{i,d(K_0)-j}(K_0) = 0+0 = 0 \quad (i+d-j>d)\\
h_{i,d(K)-i}(K)&= 0+\binom{p(K_0)}{i} = \binom{p(K)}{i} \\
h_{0,d(K)-1}(K)&= h_{0,d(K_-)}(K_-) + h_{0,d(K_0)-1}(K_0) = 1+ m(K_{0}) = m(K)\\
h_{0,d(K)-2}(K)&= h_{0,d(K_-)-1}(K_-) + h_{0,d(K_0)-2}(K_0)\\
&= m(K_-) + \frac{(m(K_0)-1)(m(K_0)-2)}{2}+p(K_{0})-1 \\
&= m(K)-2 +  \frac{(m(K)-2)(m(K)-3)}{2}+p(K)-1 \\
&= \frac{(m(K)-1)(m(K)-2)}{2}+p(K)-1.
\end{align*}
As for the assertion (e), 
\begin{align*}
h_{1,d(K)-2}(K) & = h_{1,d(K_-)-1}(K_-) + h_{1,d(K_0)-2}(K_0) = \binom{p(K_-)}{1} + h_{1,d(K_0)-2}. 
\end{align*}
Hence
\[ p(K)(m(K)-2)=p(K_-)+ (m(K_0)-2)p(K_0) \leq h_{1,d(K)-2}(K)\]
and
\[  h_{1,d(K)-2}(K) \leq p(K_-)+ (m(K_0)-2)p(K_0) +m(K_0) \leq (m(K)-2)p(K) +m(K) -1 \]
Finally, for the assertion (f),
\begin{align*}
&h_{0,d(K)-3}(K)-h_{1,d(K)-2}(K) \\
&\qquad= h_{0,d(K_-)-2}(K_-) - h_{1,d(K_-)-1}(K_-) +h_{0,d(K_0)-3}(K_0) - h_{1,d(K_0)-2}(K_0)\\
&\qquad= \frac{(m(K_-)-1)(m(K_-)-2)}{2}+p(K_-)-2 - \binom{p(K_-)}{1}\\
&\qquad \qquad+ \frac{(m(K_0)-1)(m(K_0)-2)(m(K_0)-6)}{6} - 2(p(K_0)-1)\\
&\qquad= \frac{(m(K)-3)(m(K)-4)}{2}-2 +\frac{(m(K)-2)(m(K)-3)(m(K)-7)}{6} - 2(p(K)-1)\\
&\qquad= \frac{(m(K)-1)(m(K)-2)(m(K)-6)}{6}-2(p(K)-1).
\end{align*}

\noindent
\underline{\bf Case 2: $a_0 = 3$}\\

\noindent
{\it Case 2-1: $k>0$}\\

In this case $s(K)=s(K_-)=s(K_0)$ and $p(K)=p(K_-)=p(K_0)$. 
All the assertions follow from the same induction arguments as in Case 1-1.\\

\noindent
{\it Case 2-2: $k=0$}\\

In this case  $s(K)=s(K_{0})=s(K_{-})$ and $p(K)=p(K_{0})=p(K_{-})+1$.
All the assertions except the lower bound for (e) follow from almost the same standard induction arguments as in Case 1-1.

As for the inequality (e), we need an additional argument since usual induction argument only yields a weaker inequality;
\begin{align*}
h_{1,d(K)-2}(K) &= h _{1,d(K_0)-2}(K_{0})+p(K_{-}) \geq (m(K_0)-2)p(K_0)+p(K_{-})\\
& = (m(K)-3)p(K)+p(K)-1 = (m(K)-2)p(K)-1. 
\end{align*}

Recall that $K_0$ is the closure of $\sigma_j^{2}A_0B_0$.
Let $K_{00}$ and $K_{0-}$ be the closure of braids $\sigma_jA_0B_0$ and $A_0B_0$, respectively. The skein triple $(K_0,K_{0-},K_{00})$ is a situation of Case 3-1 below and 
\[ p(K_{0-})=p(K_{0})-1 = p(K)-1,\ m(K_{0-})=m(K_0)-1=m(K)-2.\]
By applying the skein relation twice, we get
\begin{align*}
h_{1,d(K)-2}(K) &= h_{1,d(K_0)-2}(K_{0})+p(K_{-}) = h_{1,d(K_{00})-2}(K_{00})+p(K_{0-})+m(K_{0-})+p(K_{-})\\
& = h_{1,d(K_{00})-2}(K_{00})+2p(K)+m(K)-4.
\end{align*}
Thus by induction
\begin{align*}
h_{1,d(K)-2}(K) & \geq (m(K_{00})-2)p(K_{00})+2p(K)+m(K)-4
\\
& =(m(K)-4)(p(K)-1)+2p(K)+m(K)-4 \\
& =(m(K)-2)p(K)
\end{align*}
as desired. \\

\noindent
\underline{\bf Case 3: $a_0 = 2$}\\

\noindent
{\it Case 3-1: $k=0$}\\

In this case $s(K)=s(K_{0})=s(K_{-})-1$ and $p(K)=p(K_{0})+1=p(K_{-})+1$. In this case the skein relation is
\[ \widetilde{P}_{K}(\alpha,z) = (\alpha\!+\!1)\widetilde{P}_{K_-}(\alpha,z) +z^{2\delta}\widetilde{P}_{K_0}(\alpha,z) \]
hence we get
\[ h_{i,j}(K)=j_{i-1,j}(K_-)+h_{i,j}(K_-)+h_{i,j-\delta}(K)\]
Therefore
\[ d_{i,d(K)-j}(K)=h_{i-1,d(K_-)+1-j}(K_-) + h_{i,d(K_-)+1-j}(K_-) + h_{i,d(K_0)}(K_0)\]

All the assertion follows from the standard induction arguments which are similar to the argument in the Case 1-1.

Here for reader's convenience, we attach a proof for the most complicated asssertion (e);
Since $h_{1,d(K)-2}(K)= h_{0,d(K_-)-1}(K_-) +  h_{1,d(K_-)-1}(K_-) + h_{1,d(K_0)-2}(K_-)$ 
\begin{align*}
h_{1,d(K)-2} & \leq m(K_-) + p(K_-) +(m(K_0)-2)p(K_0)+m(K_0)\\
&= m(K)-1 + p(K)-1 + (m(K)-3)(p(K)-1)+m(K) -1 \\
&= (m(K)-2)p(K) + m(K)
\end{align*}
and
\begin{align*}
h_{1,d(K)-2} & \geq m(K_-) + p(K_-) +(m(K_0)-2)p(K_0)\\
&= m(K)-1 + p(K)-1 + (m(K)-3)(p(K)-1) \\
&= (m(K)-2)p(K) +1\\ 
&> (m(K)-2)p(K).
\end{align*}

\noindent
{\it Case 3-2: $k=1$}\\

In this case, thanks to the minimality of $j$ and the maximality of $a_0$ for the braid representative (\ref{eqn:form}), the braid $\beta$ should be one of the following forms;

\begin{claim}
\label{claim:key}
If $k=1$ and $a_0=2$, then $\beta$ is one of the following form.
\begin{itemize}
\item[(A)] $\beta=\sigma_{1}^2 B_0 \sigma_{1}B_1$ 
\item[(B)] $\beta=\sigma_{1}^2 B_0 \sigma_{1}^{2}B_1$
\item[(C)] $\beta=\sigma_{2}^2\sigma_1 B_0 \sigma_{2}^{2}\sigma_1 B_1$
\end{itemize}
\end{claim}
The proof of claim will be given in the next section. \\

In the case (A), $s(K)=s(K_-)=s(K_0)$ and $p(K)=p(K_0)=p(K_-)$.
In the case (B), $s(K)=s(K_-)=s(K_0)$ and $p(K)=p(K_0)=p(K_-)-1$.
In both cases, all the assertions follow from almost the same induction argument as in Case 1-1.

In the case (C), $s(K)=s(K_-)=s(K_0)$ and $p(K)=p(K_-)=p(K_-)-2$. All the assertions except the upper bound of the inequality (e) follow from induction.
As for the inequality (e), we need an additional argument since induction yields a weaker inequality
\begin{align*}
h_{1,d(K)-2}(K) &= h _{1,d(K_0)-2}(K_{0})+p(K_{-}) \leq (m(K_0)-2)p(K_0)+m(K_0)+p(K_{-})\\
& = (m(K)-3)p(K)+m(K)-1+p(K)+2\\
& = (m(K)-2)p(K)+m(K)+1. 
\end{align*}

To show the upper bound of the inequality (e), we observe that
\[ \sigma_{2}\sigma_1 B_0 \sigma_2^{2}\sigma_1 B_1 \sim  \sigma_1\sigma_{2}\sigma_1 B_0 \sigma_2^{2} B_1 = \sigma_{2} \sigma_1 \sigma_{2}B_0 \sigma_2^{2} B_1. \]
Here we denote $X \sim Y$ if braids $X$ and $Y$ are conjugate. 
Thus $K_{0}$ is represented as a closure of the positive $(n-1)$-braid $\sigma_1^{2} B'_0 \sigma_1^{2} B'_1$. Here $B'_i$ is a braid (word) obtained by shifting the indices by $1$ (namely we replace $\sigma_i$ in $B_0$, $B_1$ with $\sigma_{i-1}$).

Let $K_{0-}, K_{00}, K_{000}, K_{00-}$ be the closure of braids
\[ B'_0 \sigma_1^{2} B'_1, \sigma_1B'_0 \sigma_1^{2} B'_1, \sigma_1 B'_0 \sigma_1 B'_1, \sigma_1 B'_0 B'_1\]
respectively. By Theorem \ref{theorem:pbraid-visible}, $p(K)=p(K_{00-})=p(K_{000})=p(K_{0-})-1=p(K_{-})-2$. By induction
\[ h_{1,d(K_{000})-2} \leq (m(K_{000})-2)p(K_{000})+m(K_{000}) = (m(K)-5)p(K)+m(K)-3.\]
By skein relation, 
\begin{align*}
h_{1,d(K)-2}(K) &= h_{1,d(K_0)-2}(K_{0})+p(K_{-})\\
& = h_{1,d(K_{00})-2}(K_{00})+p(K_{0-})+p(K_{-})\\
& = h_{1,d(K_{000})-2}(K_{000})+p(K_{00-})+p(K_{0-})+p(K_{-}).
\end{align*}
Therefore
\[ h_{1,d(K)-2}(K) \leq (m(K)-5)p(K)+m(K)-3+3p(K)+3 = (m(K)-2)p(K)+m(K)\]
as desired.\\

\noindent
{\it Case 3-2: $k\geq 2$}\\

In this case $s(K)=s(K_-)=s(K_0)$.
There are two candidates $c,c'$ of decomposing circles for $K_{-}$ which is not a decomposing circle for $K$ (see Figure \ref{fig:rep2}). 

If both $c$ and $c'$ are decomposing circles of $K_-$, then $A_1,\ldots,A_{k-1}$ contains no $\sigma_{j-1}$  and $B_1,\ldots,B_{k-1}$ contains no $\sigma_{j+1}$.

This means that we may write
\[ \beta = \sigma_{j}^{2}(A_0A_1 \cdots A_{k-1})(B_0 B_1 \cdots B_{k-1}) \sigma_j^{a_1+a_2+\cdots + a_k}A_k B_k\]
Thus in this case $k=1$ and $a_1=a_2=1$. Hence this case is nothing but Case 3-1.

Therefore we may assume that at most one of $c$ and $c'$ are decomposing circles of $K_-$. In this case $p(K)=p(K_0)=p(K_-)$, or $p(K)=p(K_0)=p(K_-)-1$. In both cases almost the same standard induction argument as in Case 1-1 proves all the assertions.

\begin{figure}[htbp]
\includegraphics*[width=90mm,bb=164 489 435 715]{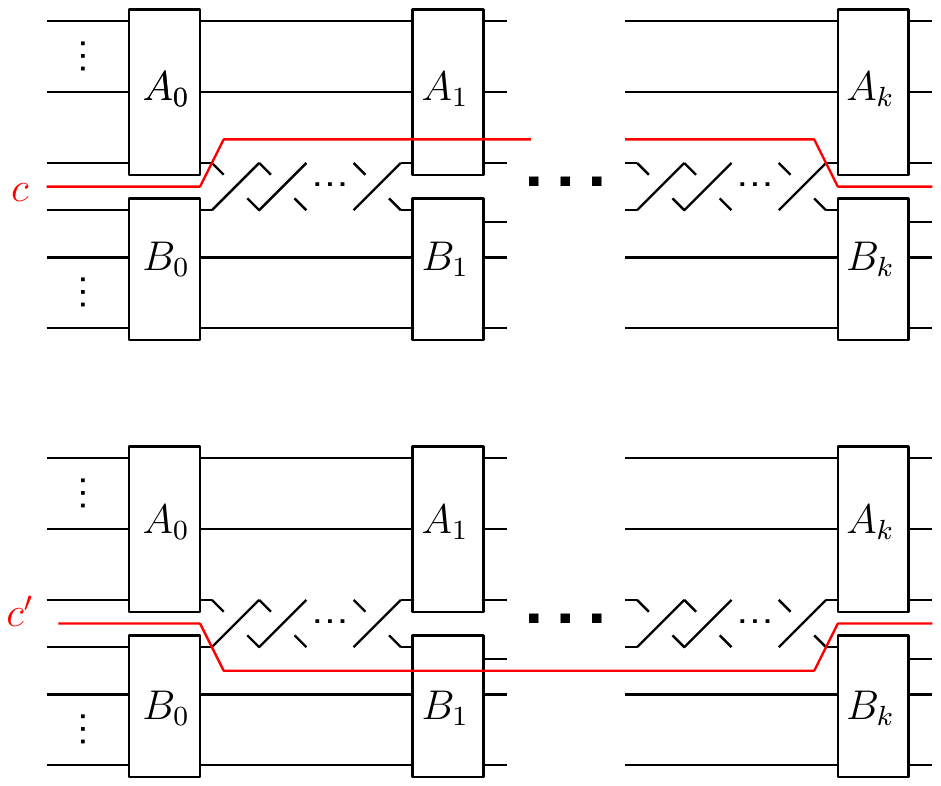}
\caption{Candidates of new decomposing circle for diagram of $K_-$ for Case 1-2. The decomposing circle $c$ (resp. $c'$) exists only if $A_1,\ldots,A_{k-1}$ contains no $\sigma_{j-1}$ (resp. $B_1,\ldots,B_{k-1}$ contains no $\sigma_{j+1}$)} 
\label{fig:rep2}
\end{figure} 

\end{proof}

\subsection{Proof of Claim \ref{claim:key}}
In this section we prove our technical assertion, Claim  \ref{claim:key}.

\begin{definition}
We say that a positive braid word is \emph{$i$-square free} if the word does not  contain $\sigma_i^{2}$. A positive braid $\beta$ is \emph{$i$-square free} if all positive braid representative of $\beta$ is $i$-square free.
\end{definition}

\begin{lemma}
\label{lemma:technical}
If a positive $n$-braid $\beta$ is $i$-square free for $i=1,\ldots,n-1$, then $\beta$ is represented by a positive braid word that contains at most one $\sigma_{n-1}$,
\end{lemma}
\begin{proof}
Let us write $\beta$ as 
\[ \beta = C_0 \sigma_{n-1} C_1 \sigma_{n-1} C_2 \sigma_{n-1}\cdots \sigma_{n-1}C_s\]
where $C_i \in \langle \sigma_1,\ldots,\sigma_{n-2}\rangle$.

Among such a word representative of $\beta$, we choose one so that $(s,\ell(C_1))$ is minimum (with respect to the lexicographical ordering).

Assume that $s > 1$. Since $\ell(C_1)$ is minimum and $A$ does not contain $\sigma_{n-1}^2$, $C_1$ is not empty and the first letter of $C_1$ should be $\sigma_{n-2}$. 

Moreover, $C_1 \neq \sigma_{n-2}$ because otherwise 
\[ \beta = C_0 \sigma_{n-1}\sigma_{n-2}\sigma_{n-1}C_2 \cdots = (C_0 \sigma_{n-2})\sigma_{n-1}(\sigma_{n-2}C_2) \cdots. \]
which contradicts the minimality of $(s,\ell(C_1))$.

Since $\beta$ is $(n-2)$ square free, the second letter of $C_1$ should be $\sigma_{n-3}$. $C_1 = \sigma_{n-2}\sigma_{n-3}$ cannot happen since this shows
\[ \beta= C_0 \sigma_{n-1}\sigma_{n-2}\sigma_{n-3}\sigma_{n-1}C_2 \cdots = C_0 \sigma_{n-1} \sigma_{n-2} \sigma_{n-1} (\sigma_{n-3}C_2) \cdots. \]

Iterating the same arguments, $C_1$ must be of the form
\[ C_1 = \sigma_{n-2}\sigma_{n-3}\cdots \sigma_2 \sigma_1 \sigma_{a} C'_1 \quad (a\neq 1)\]
for some $C'_1$, but this leads to
\begin{align*}
C_1 &= \sigma_{n-2}\sigma_{n-3}\cdots \sigma_{a+1} \sigma_{a} \sigma_{a-1} \sigma_a \sigma_{a-2} \cdots \sigma_2\sigma_1 C'_1 \\
&= \sigma_{n-2}\sigma_{n-3}\cdots\sigma_{a+1} \sigma_{a-1} \sigma_{a} \sigma_{a-1} \sigma_{a-2} \cdots \sigma_2\sigma_1 C'_1 \\
&= \sigma_{a-1} \sigma_{n-2}\cdots \sigma_2 \sigma_1 C'_1
\end{align*}
Since the first letter $\sigma_{a-1}$ can be pushed to $C_0$ across $\sigma_{j-1}$, this contradicts the minimality of $\ell(C_1)$.
Thus we conclude $s=0$ or $s=1$.
 \end{proof}

\begin{proof}[Proof of Claim \ref{claim:key}]
Let $\beta=\sigma_j^{2} A_0B_0 \sigma_{j}^{a_1}A_1B_1$, where $A_0,A_1 \in \langle \sigma_1,\ldots, \sigma_{j-1}\rangle,  B_0,B_1 \in \langle \sigma_{j+1},\ldots, \sigma_{n-1}\rangle.$ 

We assume that among such a representative of $\beta$, we take one so that $j$ is the smallest. 
If $j=1$ then $\beta$ should be of the form of (A) or (B).

Assume that $j>1$. Since we assume that $j$ is minimum, both $A_0$ and $A_1$ are $1,\ldots,j-1$ square-free. Thus by Lemma \ref{lemma:technical}, $A_0$ and $A_1$ contain at most one $\sigma_{j-1}$.

If $A_0$ or $A_1$ contains no $\sigma_{j-1}$, then the whole braid $\beta$ contains at most one $\sigma_{j-1}$. This means that we have a positive braid representative of $\beta$ of the form (\ref{eqn:form}) with $k=0$.

Thus we assume that both $A_0$ and $A_1$ contains exactly one $\sigma_{j-1}$, and we put  
\[ A_0=A'_0 \sigma_{j-1} A''_0,   A_1=A'_1 \sigma_{j-1} A''_1 \quad (A'_0,A''_0,A'_1,A''_1 \in \langle \sigma_1,\ldots,\sigma_{j-2}\rangle) \]

If $j=2$ then $A'_0,A''_0,A'_1,A''_1$ are empty word hence $\beta$ is the form of (C).
If $j\geq 3$, then
\begin{align*}
\beta &= \sigma_{j}^{2} A_0' \sigma_{j-1} A''_0 B_0 \sigma_{j}^{a_1}  A_1' \sigma_{j-1} A''_1 B_1 \\
& \sim \sigma_{j-1} A''_0 A_1' B_0 \sigma_{j}^{a_1}   \sigma_{j-1} A''_1 A_0' B_1\sigma_{j}^{2}. 
\end{align*}

Since $\beta$ is $1,\ldots,j-1$ square-free, $A''_0 A_1'$ and $A''_1 A'_0$ are $1,\ldots,j-2$ square-free.
Thus by Lemma \ref{lemma:technical}, $A''_0 A_1'$ and $A''_1 A'_0$ contains at most one $\sigma_{j-2}$. When $A_0$ or $A_1$ contains no $\sigma_{j-1}$, by the same argument we get a positive braid representative of $\beta$ of the form (\ref{eqn:form}) with $k=0$. Thus we assume that both $A''_0 A_1'$ and $A''_1 A'_0$ contain exactly one $\sigma_{j-2}$ and we put
\[ A''_0 A_1' = C \sigma_{j-2} C', \quad A''_1 A'_0 = C'' \sigma_{j-2}C''' \qquad (C,C',C'',C''' \in \langle \sigma_1,\ldots,\sigma_{j-3}\rangle).\]
Then
\begin{align*}
\beta \sim \sigma_{j-1} C \sigma_{j-2}C' B_0 \sigma_{j}^{a_1}   \sigma_{j-1} C'' \sigma_{j-2}C''' B_1\sigma_{j}^{2}\\
\sim  \sigma_{j-2}C' C'' B_0 \sigma_{j}^{a_1}   \sigma_{j-1} \sigma_{j-2}C'''C  B_1\sigma_{j}^{2}\sigma_{j-1} 
\end{align*}
Repeating the same argument,
we conclude
\begin{align*}
\beta &\sim \sigma_{1}B_0 \sigma_{j}^{a_1}  \sigma_{j-1}\sigma_{j-2}\cdots \sigma_3\sigma_2\sigma_{1}B_1 \sigma_{j}^{2}  \sigma_{j-1}\sigma_{j-2}\cdots \sigma_3 \sigma_2 \\
&=B_0 \sigma_{j}^{a_1}  \sigma_{j-1}\sigma_{j-2}\cdots \sigma_3 \sigma_{1}\sigma_2 \sigma_{1}B_1 \sigma_{j}^{2}  \sigma_{j-1}\sigma_{j-2}\cdots \sigma_3 \sigma_2 \\
&=B_0 \sigma_{j}^{a_1}  \sigma_{j-1}\sigma_{j-2}\cdots \sigma_3 \sigma_{2}\sigma_1 \sigma_{2}B_1 \sigma_{j}^{2}  \sigma_{j-1}\sigma_{j-2}\cdots \sigma_3 \sigma_2 \\
& \sim \sigma_1 (\sigma_{2}B_1 \sigma_{j}^{2}  \sigma_{j-1}\sigma_{j-2}\cdots \sigma_3 \sigma_2B_0 \sigma_{j}^{a_1}  \sigma_{j-1}\sigma_{j-2}\cdots \sigma_3 \sigma_{2})
\end{align*}
This contradicts our assumption that $\beta$ is a minimum positive braid representative.
 
\end{proof}

\section*{Appendix: Proof of basic properties of positive braid links}

\subsection*{Visiblity of primeness}

Here we give a short proof of Theorem \ref{theorem:pbraid-visible} (ii), that is a simplifcation of original Cromwell's proof is based on a fibration of positive braid links.

In \cite[1.6 Conjecture]{Cr2} Cromwell posed the conjecture that when $D$ is irreducible diagram whose canonical Seifert surface $\Sigma_D$ atatins the minimum genus Seifert surface, $D$ represents composite link if and only if $D$ is composite. Even though Ozawa proved more general result by a fairly simple argument \cite{Oz}, we hope that the proof presented here will be of independent interest toward the Cromwell's conjecture, because it uses the canonical Seifert surface in an essential way.

\begin{definition}[Murasugi sum]
Let $R_1$ and $R_2$ be oriented surfaces (with boundary). An oriented surface $R$ in $S^{3}$ is a \emph{Murasugi sum} of $R_1$ and $R_2$ if there is a 2-sphere $S\subset S^{3}$ that separates $S^{3}$ into two 3-balls $B_1$ and $B_2$, such that 
\begin{enumerate}
\item $R_1 \subset B_1$, $R_2 \subset B_2$.
\item $D:=R_1 \cap S = R_2 \cap S$ is a $2n$-gon. 
\end{enumerate} 
(See Figure \ref{fig:murasugi-sum}). We often say that $R$ is a Murasugi sum of $R_1$ and $R_2$ along the $2n$-gon $D$.
\end{definition}

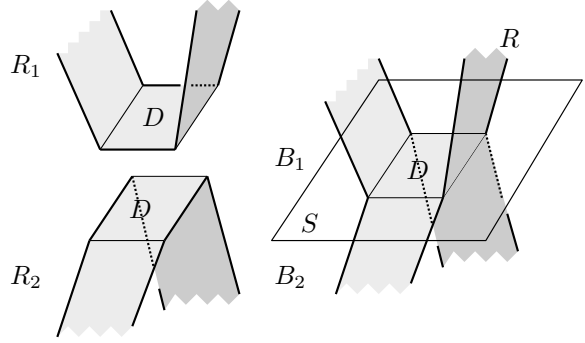
\begin{figure}[htbp]
\includegraphics*[width=70mm]{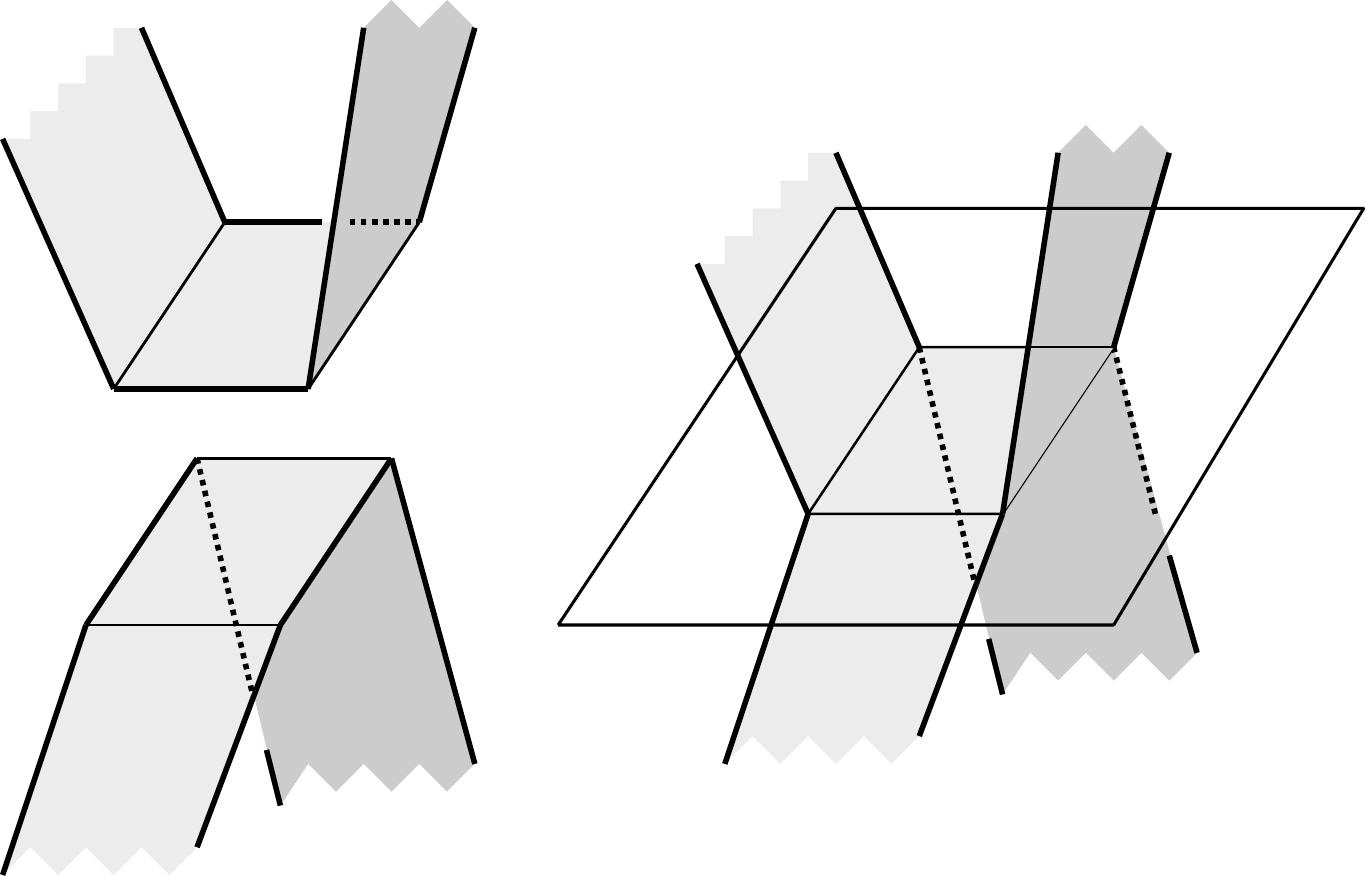}
\begin{picture}(0,0)
\put(-175,45) {$D$}
\put(-170,80) {$D$}
\put(-220,20) {$R_2$}
\put(-220,100) {$R_1$}
\put(-110,40) {$S$}
\put(-120,65) {$B_1$}
\put(-120,20) {$B_2$}
\put(-35,110) {$R$}
\put(-70,60) {$D$}
\end{picture}
\caption{Murasugi sum $R$ of $R_1$ and $R_2$.} 
\label{fig:murasugi-sum}
\end{figure} 

Among various nice properties of Murasugi sum, we use the following.

\begin{theorem}[Stallings \cite{St}, Gabai \cite{Ga}]
Assume that $K_i= \partial R_i$ is a fibered link with fiber $R_i$ $(i=1,2)$. If $R$ is a Murasugi sum of $R_1$ and $R_2$, then $K := \partial R$ is a fibered link with fiber $R$. Moreover, the monodromy $\phi :R \rightarrow R$ is given by $\phi=\phi_1 \phi_2$, where $\phi_{i}:R \rightarrow R$ is the monodromy of $K_i$,  viewed as a homeomorphism of $R$ by extending identity outside of $R_i$. 
\end{theorem}

We fix a positive $n$-braid (word) $\beta$ whose closure $\widehat{\beta}$ is $K$.
Applying Seifert's algorithm, we have a canonical Seifert surface $\Sigma_{\beta}$.
By construction, $\Sigma_{\beta}$ is made of two kind of pieces; the \emph{Seifert disk}, disjoint union of $n$ disks $D_1,\ldots,D_{n}$, and twisted bands connecting $i$-th and $(i+1)$-st disks that corresponds to each $\sigma_{i}$ in the braid $\beta$.

Let $n_{i}$ be the number of $\sigma_i$ in the braid word. Then the canonical Seifuert surface $\Sigma_{\beta}$ is Murasugi sum of the canonical Seifert surfaces of $(2,n_{i})$ torus knot/links (Figure \ref{fig:surface} (i)). Thus $K$ is fibered with fiber $\Sigma_{\beta}$.

The monodromy of the $(2,p)$ torus knot/link is a product of Dehn twist
\[ \phi_i = T_{c_1} \cdots T_{c_{p-1}} \]
where $c_j$ is a simple closed curve given in Figure \ref{fig:surface} (ii) and $T_c$ denotes the Dehn twist along $c$.
Hence the monodromy $\phi$ of the positive braid knot/link $K$ is written as a product of positive Dehn twists, reflecting the positivity of braids.
\[ \phi = \phi_1\phi_2\cdots \phi_{n-1} = T_{c_1}\cdots T_{c_{m}}\]
where $m$ is the $1$st betti number of $\Sigma_{\beta}$.

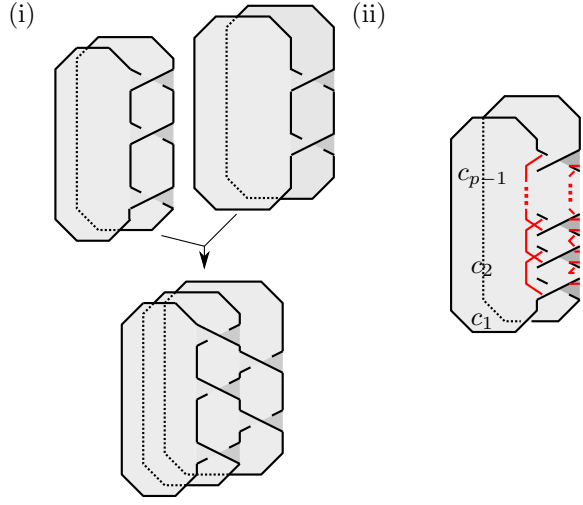
\begin{figure}[htbp]
\includegraphics*[width=70mm]{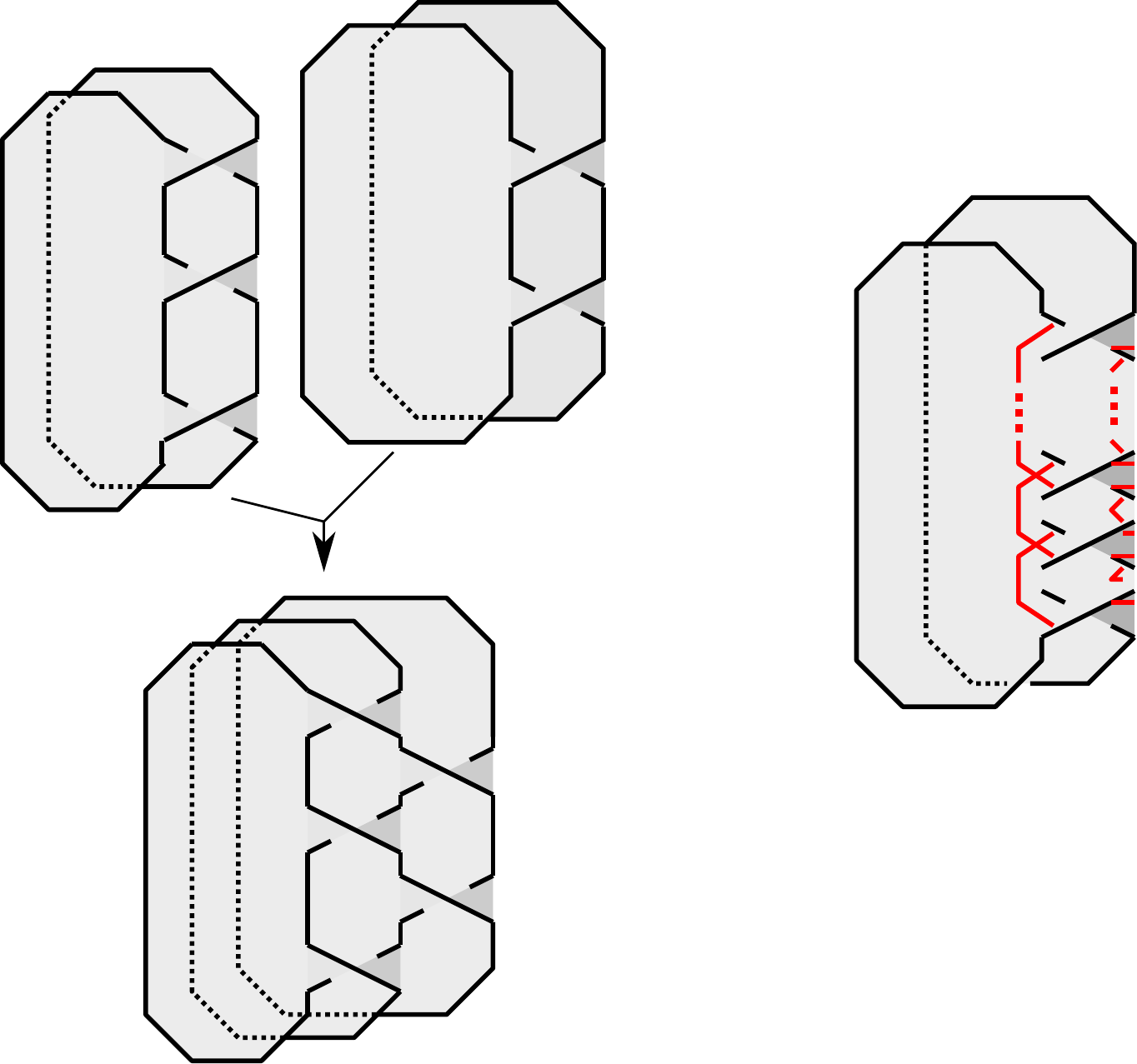}
\begin{picture}(0,0)
\put(-220,180) {(i)}
\put(-90,180) {(ii)}
\put(-45,65) {$c_1$}
\put(-45,85) {$c_2$}
\put(-50,120) {$c_{p-1}$}
\end{picture}
\caption{(i) Canonical Seifert surface $\Sigma_{\beta}$ of a positive braid $\beta$ as the Murasugi sum of canonical Seifert surface of $(2,n)$-torus knots/links. (ii) The monodromy of the  $(2,p)$-torus knots/links. } 
\label{fig:surface}
\end{figure}

\begin{lemma}
\label{lemma:App1}
Let $\gamma \in \Sigma_{\beta}$ be a properly embedded arc.
$\phi(\gamma)=\gamma$ if and only if $\gamma$ can be put so that $\gamma \cap c_i =\emptyset$ for every $c_i$. (Here $=$ means isotopic relative to the boundary).
\end{lemma}
\begin{proof}
By isotopy we put $\gamma$ so that it attains the minimum geometric intersections for all $c_i$. The `if' direction is obvious since $\gamma \cap c_i =\emptyset$ implies $T_{c_i}(\gamma)=\gamma$. We show that if $\gamma \cap c_i \neq \emptyset$ for some $i$, then $\phi(\gamma)\neq \gamma$. To see this, we use the \emph{right-veering order} $\prec_{\sf right}$ on embedded arcs.

We fix a base point $\ast \in \partial \Sigma_{\beta}$ and let $\mathcal{A}$ be the isotopy classes of oriented properly embedded arc that begins at $\ast$. For two arc $\gamma,\gamma' \in \mathcal{A}$ we define $\gamma \prec_{\sf right} \gamma'$ if near the base point $\ast$, $\gamma'$ lies on the right-hand side of $\gamma$, when $\gamma$ and $\gamma'$ are put so that they have the minimum geometric intersection. Then $\prec_{\sf right}$ defines a total ordering of $\mathcal{A}$.

By definition, this total ordering is invariant under the action of mapping class group of $\Sigma_{\beta}$; for $\theta \in MCG(\Sigma_{\beta})$, $\gamma \prec_{\sf right} \gamma'$ implies $\theta(\gamma) \prec_{\sf right} \theta(\gamma')$. Also, the right-handed Dehn twist $T_c$ have the property that $\gamma \preceq_{\sf right} T_{c}(\gamma)$ for all $\gamma \in \mathcal{A}$.
Now we assume that if $\gamma \cap c_i \neq \emptyset$ for some $i$, among such $i$, we take maximum one so if $j>i$ then  $\gamma \cap c_j =\emptyset$. 

Then $T_{c_i}(\gamma) \succ_{\sf right} \gamma$, hence
\begin{align*}
\phi(\gamma) &= T_{c_1} \cdots T_{c_{i-1}}T_{c_i} T_{c_{i+1}}\cdots T_{c_{\ell-(n-1)}}(\gamma) =  T_{c_1} \cdots T_{c_{i-1}}(T_{c_i}(\gamma))\\
& \succ_{\sf right} T_{c_1} \cdots T_{c_{i-1}}(\gamma)\\
& \succeq_{\sf right} \gamma
\end{align*}
so $\phi(\gamma) \neq \gamma$. 
\end{proof}

\begin{proof}[Proof of Theorem \ref{theorem:pbraid-visible}(ii)]
Assume that $K$ is composite. Then the intersection of a decomposing sphere and $\Sigma_{\beta}$ gives rise to a non-boundary parallel, properly embedded arc $\gamma$ such that $\phi(\gamma)=\gamma$. 
By lemma \ref{lemma:App1}, by isotopy we assume that $\gamma \cap c_{i} = \emptyset$ for all $c_i$.

Assume that one of the endpoint of $\gamma$ lies on the $i$-th Seifert disk $D_{i}$.  From the description of the monodromy $\phi$ and curves $c_i$, by pushing the arc $\gamma$ onto the boundary, we find an another (non-boundary-parallel) arc $\gamma'$ contained in $D_i$ such that $\gamma' \cap c_i=\emptyset$ for all $i$. This implies $\gamma'$ separates $D_{i}$ into two pieces so that one piece contains all the twisted bands attached to $D_{i-1}$, and the other piece contains all twisted bands attached to $D_{i+1}$. By adding an arc that is parallel to $\partial D_i$ we get a decomposing circle of $D$ (Figure \ref{fig:decompose} (ii)).

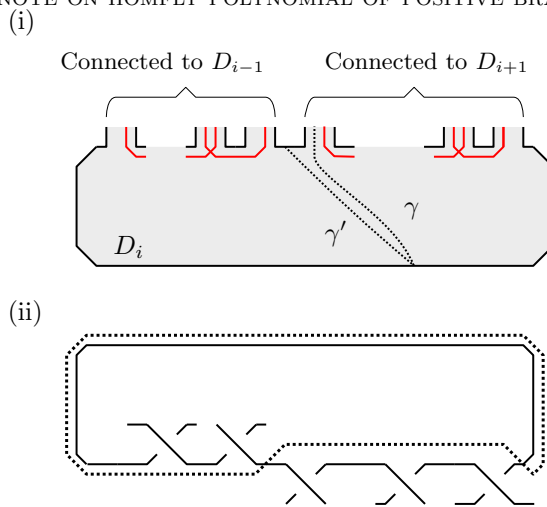
\begin{figure}[htbp]
\vspace{0.6cm}
\includegraphics*[width=65mm]{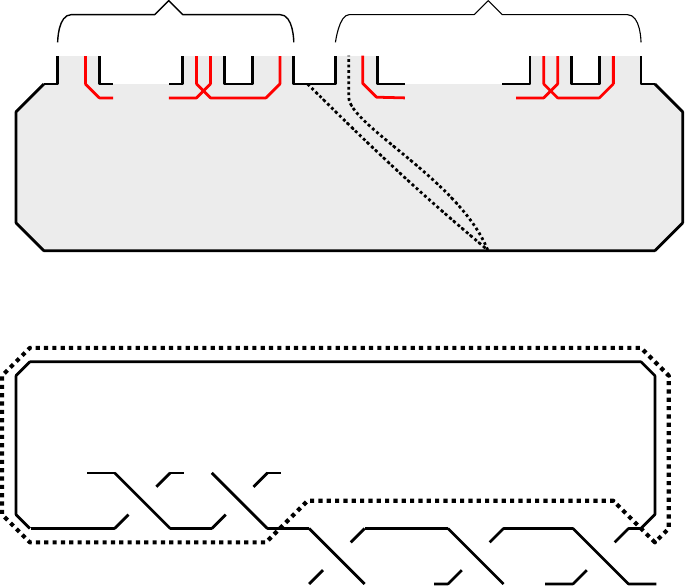}
\begin{picture}(0,0)
\put(-210,180) {(i)}
\put(-210,70) {(ii)}
\put(-190,165) {\small Connected to $D_{i-1}$}
\put(-90,165) {\small Connected to $D_{i+1}$}
\put(-170,95) {$D_i$}
\put(-60,110) {$\gamma$}
\put(-90,100) {$\gamma'$}
\end{picture}
\caption{(i) An arc $\gamma$ disjoint from $c_i$ can be taken so that it is contained in a single Seifert disk $D_i$. (ii) The existence of an arc $\gamma'$ shows that the diagram is composite.} 
\label{fig:decompose}
\end{figure} 

\end{proof}

\subsection*{Skein resolution in the realm of positive braid links}

As we mentioned earlier, Lemma \ref{lemma:technical} implies Theorem \ref{theorem:resolution}.

\begin{proof}[Proof of Theorem \ref{theorem:resolution}]
We prove the theorem by induction on $n$, where $n$ is the minimum number of strands that is needed to represent $K$ as the closure of a positive braid $\beta$.

Assume, to the contrary that $\beta$ is $i$-square free for all $i$. Then Lemma \ref{lemma:technical} shows that $\beta$ is written so that it contains at most one $\sigma_{n-1}$. Since $\beta$ is minimum positive braid representative, $\beta$ contains no $\sigma_{n-1}$ so $K$ is the split union of other positive braid link that is a closure of an $(n-1)$ positive braid and the unknot, so by induction, $K$ should be unlink. 
\end{proof}

\end{document}